\documentclass[11pt]{article}
\usepackage{latexsym,amsmath,amssymb,mathrsfs,graphicx,hyperref,mathpazo,soul,color,amsthm}
\hypersetup{pdftex,colorlinks=true,linkcolor=blue,citecolor=blue,filecolor=blue,urlcolor=blue}

\usepackage{url}
\usepackage{xcolor}
\usepackage{soul}

\def\tr{{\raise0pt\hbox{$\scriptscriptstyle\top$}}}

\newtheorem{theorem}{Theorem}[section]
\newtheorem{lemma}[theorem]{Lemma}

\newtheorem{proposition}[theorem]{Proposition}

\newtheorem{problem}[theorem]{Problem}

\theoremstyle{remark}
\newtheorem*{remark}{Remark}
\theoremstyle{plain}

\title{\vspace{-1.5 cm}{\bf On the shortest open cubic equations}}

\author{
Bogdan Grechuk\footnote{School of Computing and Mathematical Sciences, University of Leicester, LE1 7RH, UK; bg83@leicester.ac.uk},
\quad
Ashleigh Ratcliffe\footnote{School of Computing and Mathematical Sciences, University of Leicester, LE1 7RH, UK; amw64@leicester.ac.uk}
}

\begin{document}
	
	\maketitle
	
\begin{abstract}
We use cubic reciprocity to prove that the equation $7x^3+2y^3=3z^2+1$ has no integer solutions. Prior to this work, it was the shortest cubic equation for which the existence of integer solutions remained open. We conclude with a list of the new shortest open cubic equations.
\end{abstract}

\section{Introduction}

The tenth problem in Hilbert's famous list \cite{MR1557926} asks whether there exists a procedure that can determine, for any multivariate polynomial $P$ with integer coefficients, whether the equation $P=0$ has an integer solution. In 1970, Matiyasevich \cite{MR258744}, building on earlier work of Davis, Putnam, and Robinson \cite{MR133227}, proved that no such general algorithm exists.

Moreover, this undecidability persists even when attention is restricted to equations of degree four, at least for positive integer solutions \cite{MR578379}. By contrast, all linear and quadratic equations can be handled effectively \cite{MR2055712}. For cubic equations, however, the problem remains open: while cubic equations in at most two variables can be decided algorithmically \cite{MR4731077}, for cubic equations in $n\geq 3$ variables there is neither a known algorithm for determining the existence of integer solutions nor a proof that no such algorithm can exist.

\begin{problem}\label{prob:cub}
	Given a cubic polynomial $P(x_1,\dots,x_n)$ with integer coefficients, determine whether the equation
	\begin{equation}\label{eq:gencub}
		P(x_1,\dots,x_n) = 0
	\end{equation}
	has an integer solution.
\end{problem}

Given the difficulty of Problem \ref{prob:cub} in general, we initiated a project to solve it for individual equations arranged in a certain order. Specifically, if a polynomial $P$, presented in reduced form (that is, no two monomials differ only by a constant factor), consists of monomials of degrees $d_1, \dots, d_k$ with coefficients $a_1, \dots, a_k$, then the \emph{length} of equation \eqref{eq:gencub} is defined by
\begin{equation}\label{eq:lengthdef}
	l(P) = \sum_{i=1}^k \log_2(|a_i|) + \sum_{i=1}^k d_i = \log_2 L(P),
	\quad \text{where} \quad L(P) = \prod_{i=1}^k \bigl(|a_i|\cdot 2^{d_i}\bigr).
\end{equation}
For example, the equation
\begin{equation}\label{eq:6x2zpz2yp3y3p1}
	6x^2z+z^2y+3y^3+1=0
\end{equation}
has length
$$
l = (\log_2 6 + 2 + 1) + (2 + 1) + (\log_2 3 + 3) + \log_2 1 = 10+2\log_2 3 \approx 13.2.
$$
Following \cite[Section~1.1.1]{MR4823864}, we regard two equations \eqref{eq:gencub} as equivalent if they can be transformed into each other by permuting or renaming the variables, substitutions of the form $x_i \to -x_i$, or multiplication by $-1$. Then for any given $B>0$, there are only finitely many (equivalence classes of) equations \eqref{eq:gencub} of length $l \leq B$. This allows us to order all equations of the form \eqref{eq:gencub} by length, and attempt to solve Problem \ref{prob:cub} in this order. The shortest cubic equations either have easily found solutions, such as $xyz+1=0$, or have no integer solutions for trivial reasons, such as $2xyz+1=0$; such elementary cases can be eliminated automatically by a computer program. Equation $x^2y+2y+1=0$ of length $l=5$ has solutions modulo every positive integer \cite[Proposition 7.2]{MR4823864}, but no integer solutions for the trivial reason that $-1<y=-(x^2+2)^{-1}<0$ for every integer $x$. A more interesting example is the equation $y^2=x^3-3$, which has length $2+3+\log_2 3 \approx 6.6$. This equation defines an elliptic curve, and the nonexistence of integer points can be proved by standard methods \cite[Proposition 7.4]{MR4823864}. More generally, Problem \ref{prob:cub} is algorithmically solvable for all two-variable cubic equations \cite{MR4731077}.

The shortest cubic equation in $n\geq 3$ variables that has solutions modulo every positive integer but no integer solutions is
$
x^2 y + y^2 + z^2 + 1 = 0,
$
which has length $l=7$. This equation can be handled using the law of quadratic reciprocity \cite[Section 4.1.7]{MR4823864}. Repeating this solve--exclude strategy leads naturally to more subtle examples. One such example is
$$
y^2-xyz+z^2 = x^3-5,
$$
of length $l=10+\log_2 5 \approx 12.3$, popularly known as the ``Fruit Diophantine equation.'' It was solved in \cite{majumdar2023107} and has since led to several related works \cite{vaishya2022class, prakash2024generalized, sahoo2024generalized}.

In \cite[Section~8.3]{MR4823864}, we reported the resolution of Problem \ref{prob:cub} for all cubic Diophantine equations of length $l \leq 10+2\log_2 3 \approx 13.2$, except for \eqref{eq:6x2zpz2yp3y3p1}. In fact, \cite{MR4823864} resolved Problem \ref{prob:cub} for all cubic equations of length $l < 12+\log_2 3 \approx 13.6$, except for \eqref{eq:6x2zpz2yp3y3p1} and the equations
\begin{equation}\label{eq:y2p10xyzpx3mxm2}
	y^2+10xyz+x^3-x-2=0
\end{equation}
\begin{equation}\label{eq:y2p7xyzp3x3m2}
	y^2+7xyz+3x^3-2=0
\end{equation}
and
\begin{equation}\label{eq:7x3p2y3m3z2m1}
	7x^3+2y^3=3z^2+1.
\end{equation}
Equation \eqref{eq:y2p10xyzpx3mxm2} has length $11+\log_2 5 \approx 13.3$, while equations \eqref{eq:y2p7xyzp3x3m2} and \eqref{eq:7x3p2y3m3z2m1} have length $l=9+\log_2 21 \approx 13.4$.

For equations \eqref{eq:6x2zpz2yp3y3p1}--\eqref{eq:y2p7xyzp3x3m2}, searches for small solutions returned no results. Heuristics nevertheless suggested that the number of integer solutions $(x,y,z)$ satisfying
\begin{equation}\label{eq:boundB}
\max\{|x|,|y|,|z|\}\leq B
\end{equation}
should grow like $c \log B$ for a suitable constant $c$, which motivates the search for larger solutions. And indeed, in 2024 large integer solutions were found for \eqref{eq:6x2zpz2yp3y3p1}--\eqref{eq:y2p7xyzp3x3m2}; see \cite{grechuk2026cubic} for the methods used to locate such solutions. Equation \eqref{eq:7x3p2y3m3z2m1} therefore became the shortest cubic equation for which the existence of integer solutions remained open. This equation is of the form $ax^3+by^3=cz^2+d$, and for such equations heuristics suggest that either the number of solutions satisfying \eqref{eq:boundB} grows polynomially in $B$, or no integer solutions exist at all. Since searches for small solutions returned no results, this provided strong evidence that equation \eqref{eq:7x3p2y3m3z2m1} has no integer solutions, but the proof had remained elusive until now. Recently, Alfaraj \cite{alfaraj2025brauer} showed how to use the Brauer--Manin obstruction to prove the non-existence of integer solutions to certain equations of similar form, including, for example, equation $9x^3+y^3=z^2+3$, but not equation \eqref{eq:7x3p2y3m3z2m1}.

Equation \eqref{eq:7x3p2y3m3z2m1} was listed as open in the monograph \cite{MR4823864}, in the collection of open problems \cite{grechuk2024systematic}, and on MathOverflow \url{https://mathoverflow.net/q/467988/}, where Denis Shatrov reduced the problem to a conjecture that, under certain conditions, $3z^2+1$ has a prime factor $p\neq 7$ such that $28$ is a cubic non-residue modulo $p$. 

In this paper, we prove Shatrov's conjecture by cubic reciprocity and thereby show that equation \eqref{eq:7x3p2y3m3z2m1} has no integer solutions. The resulting contradiction can also be interpreted as an integral Brauer--Manin obstruction, but we formulate the argument in terms of cubic reciprocity in order to keep the prerequisites minimal.

The paper is organized as follows. Section \ref{sec:review} reviews some standard facts from the theory of Eisenstein integers and cubic reciprocity that will be needed later. Section \ref{sec:main} proves that equation \eqref{eq:7x3p2y3m3z2m1} has no integer solutions. Section \ref{sec:next} lists the new shortest open cubic equations after equation \eqref{eq:7x3p2y3m3z2m1} is resolved.

\section{Eisenstein integers and cubic reciprocity}\label{sec:review}

In this section we review some standard facts about Eisenstein integers and cubic reciprocity that will be needed later.

Let $\omega = \frac{-1+\sqrt{-3}}{2}$. Then $\omega^2=-1-\omega$ and $\omega^3=1$. The ring $D=\mathbb{Z}[\omega]=\{a+b\omega \mid a,b \in \mathbb{Z}\}$ is the ring of Eisenstein integers. We use the standard divisibility notation $\beta \mid \alpha$ and $\alpha \equiv \gamma \pmod \beta$. The units of $D$ are $\pm1$, $\pm \omega$, and $\pm \omega^2$; two non-zero elements $\alpha$ and $\beta$ are \emph{associates} if $\alpha=\beta u$ for some unit $u$, and are \emph{coprime} if every common divisor is a unit.

The \emph{norm} of $\alpha \in D$ is defined by
\begin{equation}\label{def:norm}
	N(\alpha) = N(a+b\omega) := \alpha \bar{\alpha} = a^2-ab+b^2,
\end{equation}
where $\bar{\alpha}= a+b\omega^2=(a-b)-b\omega$ is the complex conjugate of $\alpha=a+b \omega$. We note that $N(\alpha)$ is a positive integer for every $\alpha\neq 0$, and that $N(\alpha \beta) = N(\alpha)N(\beta)$ for all $\alpha,\beta \in D$. In particular, $\beta \mid \alpha$ implies $N(\beta) \mid N(\alpha)$, and $N(u)=1$ for every unit $u$.

A nonzero nonunit $\pi \in D$ is \emph{prime} if $\pi \mid \alpha\beta$ implies $\pi \mid \alpha$ or $\pi \mid \beta$. Every non-zero element $\alpha \in D$ factors uniquely into primes up to associates \cite[Chapter 1]{modern}, and we denote by $v_\pi(\alpha)$ the total exponent of all associates of a prime $\pi$ in the prime factorization of $\alpha$, with the convention $v_\pi(0)=\infty$.

Let $\pi$ be a prime in $D$ such that $N(\pi)\neq 3$. Then $N(\pi)\equiv 1 \pmod 3$, and for any $\alpha \in D$, either $\pi \mid \alpha$, or there is a unique integer $m\in\{0,1,2\}$ such that $\alpha^{(N(\pi)-1)/3} \equiv \omega^m \pmod \pi$ \cite[Proposition 9.3.2]{modern}. The cubic residue character of $\alpha$ modulo $\pi$ is then defined by
\begin{equation}\label{def:cubicresidue}
	\left(\frac{\alpha}{\pi}\right)_3 =
	\begin{cases}
		0, & \text{if } \pi \mid \alpha, \\
		\gamma, & \text{where } \alpha^{(N(\pi)-1)/3} \equiv \gamma \pmod \pi,\ \gamma \in \{1,\omega,\omega^2\}.
	\end{cases}
\end{equation}

The following proposition summarizes the basic properties of the cubic residue character.
\begin{proposition}\cite[Proposition 9.3.3]{modern}\label{prop:cubicres}
	\begin{itemize}
		\item[(a)] For $\pi \nmid \alpha$, $\left(\frac{\alpha}{\pi}\right)_3 =1$ if and only if the congruence $x^3 \equiv \alpha \pmod \pi$ is solvable, that is, if and only if $\alpha$ is a cubic residue modulo $\pi$.
		\item[(b)] $\alpha^{(N(\pi)-1)/3} \equiv \left(\frac{\alpha}{\pi}\right)_3 \pmod \pi$.
		\item[(c)] $\left(\frac{\alpha\beta}{\pi}\right)_3 =\left(\frac{\alpha}{\pi}\right)_3 \left(\frac{\beta}{\pi}\right)_3$.
		\item[(d)] If $\alpha \equiv \beta \pmod \pi$, then $\left(\frac{\alpha}{\pi}\right)_3 = \left(\frac{\beta}{\pi}\right)_3$.
	\end{itemize}
\end{proposition}

We say that an element $\gamma \in D$ is \emph{primary} if $\gamma \equiv 2 \pmod 3$. The law of cubic reciprocity for primary primes is formulated as follows \cite[Theorem 7.8]{lemmermeyer2000reciprocity}.

\begin{theorem}\label{th:cubicrecip}
	Let $\pi_1$ and $\pi_2$ be primary primes. Then
	$
	\left(\frac{\pi_2}{\pi_1}\right)_3 =  \left(\frac{\pi_1}{\pi_2}\right)_3.
	$
\end{theorem}

If $\pi \in D$ is coprime to $3$, then exactly one of its associates is primary \cite[Lemma 3.5.8]{wewers2014algebraische}. In other words, every element $\pi \in D$ coprime to $3$ has a unique \emph{primary associate}. This implies that every primary $\eta \in D$ can be written uniquely in the form\footnote{Note that after replacing each prime factor by its unique primary associate, the remaining unit can only be $\pm 1$, not $\pm \omega$ or $\pm \omega^2$. This is because both $\eta$ and each $\pi_i$ are congruent to a rational integer modulo $3$, and hence the same must be true of the remaining unit.}
$$
\eta=\pm \prod_i \pi_i^{e_i},
$$
where each $\pi_i$ is a primary prime and each $e_i$ is a non-negative integer exponent. This allows us to define, for every primary $\eta \in D$ and every $\alpha \in D$,
\begin{equation}\label{def:chi}
	\left(\frac{\alpha}{\eta}\right)_3 :=\prod_i \left(\frac{\alpha}{\pi_i}\right)_3^{e_i}.
\end{equation}
This symbol is multiplicative. The statement of Proposition \ref{prop:cubicres}(a) continues to hold in one direction: if $\alpha$ and $\eta$ are coprime, then the symbol is still guaranteed to be $1$ whenever $\alpha$ is a cubic residue modulo $\eta$, although the converse need not hold. The general law of cubic reciprocity for arbitrary primary elements is then formulated as follows\footnote{In \cite[Theorem 7.2.4]{classical}, Theorem \ref{th:gencubic} is stated under the additional assumption that $\eta$ and $\theta$ are coprime. If they are not coprime, however, the statement is trivially true because both $\left(\frac{\eta}{\theta}\right)_3$ and $\left(\frac{\theta}{\eta}\right)_3$ are equal to $0$.} \cite[Theorem 7.2.4]{classical}.

\begin{theorem}\label{th:gencubic}
	Let $\eta,\theta \in D$ be primary elements. Then
	$
	\left(\frac{\eta}{\theta}\right)_3 = \left(\frac{\theta}{\eta}\right)_3.
	$
\end{theorem}

\section{The resolution of the shortest open equation}\label{sec:main}

This section is devoted to the proof of the following theorem, which is the main result of the paper.

\begin{theorem}\label{th:main}
	Equation \eqref{eq:7x3p2y3m3z2m1} has no integer solutions.
\end{theorem}

We start by proving several lemmas.

\begin{lemma}\label{lem:ab}
Suppose that $(x,y,z)\in\mathbb Z^3$ satisfies \eqref{eq:7x3p2y3m3z2m1}. Then $3\nmid xyz$, $z$ is even, $z \equiv \pm 3 \pmod 7$, $7\mid y$, and 
\[
3z^2+1=7^a b,\qquad 7\nmid b,
\]
for some integers $a,b$ such that $a\geq 1$, $a\equiv 0$ or $1\pmod 3$, $b>1$, and
\[
a\equiv 1\pmod 3 \Longrightarrow b\equiv \pm 1\pmod 7,\qquad
a\equiv 0\pmod 3 \Longrightarrow b\equiv \pm 2\pmod 7.
\]
\end{lemma}
\begin{proof}
	Reductions of \eqref{eq:7x3p2y3m3z2m1} modulo $9$ and $8$ show that $3\nmid xyz$ and $z$ is even, while reduction modulo  $7$ shows that $z \equiv \pm 3 \pmod 7$ and $7\mid y$. 
	Write
	\[
	x=7^u x_0,\qquad y=7^v y_0,\qquad 7\nmid x_0y_0.
	\]
	Since $7\mid y$, we have $v\geq 1$. Then
	\[
	7^a b=3z^2+1=7x^3+2y^3=7^{1+3u}x_0^3+7^{3v}(2y_0^3).
	\]
	The exponents $1+3u$ and $3v$ are distinct, so
	\[
	a=\min\{1+3u,3v\}.
	\]
	Hence $a\equiv 0$ or $1\pmod 3$. If $a=1+3u<3v$, then $b\equiv x_0^3\equiv \pm 1\pmod 7$; if $a=3v<1+3u$, then $b\equiv 2y_0^3\equiv \pm 2\pmod 7$. 
	Finally, if $b=1$, then $3z^2+1=7^a$. Because $3\nmid z$, we have $z^2\equiv 1,4$ or $7\pmod 9$, so $3z^2+1\equiv 4\pmod 9$. But the powers of $7$ modulo $9$ are periodic with cycle $7,4,1$, forcing $a\equiv 2\pmod 3$, a contradiction.
\end{proof}

\begin{lemma}\label{lem:symbol}
Suppose that $(x,y,z)\in\mathbb Z^3$ satisfies \eqref{eq:7x3p2y3m3z2m1}, that $z\equiv 3\pmod 7$, and let $a,b$ be as in Lemma \ref{lem:ab}. Set
\[
\alpha:=1+z+2z\omega,\qquad \rho:=2+3\omega,\qquad \beta:=\alpha/\rho^a.
\]
Then $\beta \in D$ is not a unit, $N(\beta)=b>1$, and $\beta$ is coprime to $2,3,\rho,\overline{\rho}$, and $\overline{\beta}$. 
\end{lemma}
\begin{proof}
	We have
	\[
	N(\alpha)=\alpha\overline{\alpha}=(1+z)^2-(1+z)(2z)+(2z)^2=1+3z^2.
	\]
	If a prime $\pi$ divides both $\alpha$ and $\overline{\alpha}$, then $\pi\mid (\alpha+\overline{\alpha})=2$. Since $2$ is prime in $D$ and $z$ is even, we have $\alpha\equiv 1\pmod 2$, so this is impossible. Thus $\alpha$ and $\overline{\alpha}$ are coprime. We have	
	\[
	\overline{\rho}=2+3\omega^2=-1-3\omega.
	\]
	Then $7=\rho\overline{\rho}$, and $\rho$ and $\overline{\rho}$ are, up to associates, the only primes dividing $7$.
	Because $z\equiv 3\pmod 7$, we have $z\equiv 3\pmod \rho$ and $z\equiv 3\pmod{\overline{\rho}}$. Modulo $\rho$ we have $\omega\equiv 4$, hence
	\[
	\alpha\equiv 1+z+8z\equiv 1+2z\equiv 0\pmod \rho,
	\]
	so $\rho\mid\alpha$. Modulo $\overline{\rho}$ we have $\omega\equiv 2$, and therefore
	\[
	\alpha\equiv 1+z+4z\equiv 1+5z\equiv 2\pmod{\overline{\rho}}.
	\]
	In particular, $\overline{\rho}\nmid\alpha$.
	Since $N(\alpha)=3z^2+1=7^a b$ and $7\nmid b$, it follows that $v_\rho(\alpha)=a$, hence $\beta:=\alpha/\rho^a$ is an element of $D$ satisfying $\rho \nmid \beta$ and $N(\beta)=b$. By Lemma \ref{lem:ab}, $b>1$, so $\beta$ is not a unit. Next, $\overline{\rho}\nmid\alpha$ implies $\overline{\rho}\nmid\beta$, and because $\alpha\equiv 1\pmod 2$ and $3\nmid N(\beta)=b$, it is also coprime to $2$ and $3$. Finally, $\beta$ and $\overline{\beta}$ are coprime because any common divisor would divide both $\alpha$ and $\overline{\alpha}$.
\end{proof}	
	
\begin{lemma}\label{lem:symbol2}
	Suppose that $(x,y,z)\in\mathbb Z^3$ satisfies \eqref{eq:7x3p2y3m3z2m1}, that $z\equiv 3\pmod 7$, let $a,b,\rho,\alpha,\beta$ be as in Lemma \ref{lem:symbol}, and let $\beta_0$ be the primary associate of $\beta$. Then
	\[
	\left(\frac{28}{\beta_0}\right)_3=\omega^2.
	\]
\end{lemma}
\begin{proof}	
	Since $2$, $\rho$, and $\overline{\rho}$ are primary, cubic reciprocity gives
	\[
	\left(\frac{28}{\beta_0}\right)_3=\left(\frac{\beta_0}{2}\right)_3^2\left(\frac{\beta_0}{\rho}\right)_3\left(\frac{\beta_0}{\overline{\rho}}\right)_3.
	\]
	The unit group of $D$ is generated by $-\omega$. Since $-1$ is a cube and Proposition \ref{prop:cubicres}(b) gives
	\[
	\left(\frac{\omega}{2}\right)_3=\omega,
	\qquad
	\left(\frac{\omega}{\rho}\right)_3=\left(\frac{\omega}{\overline{\rho}}\right)_3=\omega^2,
	\]
	we have
	\[
	\left(\frac{-\omega}{2}\right)_3^2\left(\frac{-\omega}{\rho}\right)_3\left(\frac{-\omega}{\overline{\rho}}\right)_3=\omega^2\cdot\omega^2\cdot\omega^2=1.
	\]
	Therefore the right-hand side is unchanged if $\beta_0$ is replaced by any associate, and so
	\[
	\left(\frac{28}{\beta_0}\right)_3=\left(\frac{\beta}{2}\right)_3^2\left(\frac{\beta}{\rho}\right)_3\left(\frac{\beta}{\overline{\rho}}\right)_3.
	\]
	Because $\alpha\equiv 1\pmod 2$ and $\rho\equiv \omega\pmod 2$, we have
	\[
	\beta=\alpha\rho^{-a}\equiv \omega^{-a}=\omega^{2a}\pmod 2.
	\]
	Since $N(2)=4$, Proposition \ref{prop:cubicres}(b) implies that
	\[
	\left(\frac{\beta}{2}\right)_3=
	\begin{cases}
	1,& a\equiv 0\pmod 3,\\
	\omega^2,& a\equiv 1\pmod 3.
	\end{cases}
	\]
	Also $\rho\equiv 1\pmod{\overline{\rho}}$, so $\beta\equiv \alpha\equiv 2\pmod{\overline{\rho}}$. Because $2\equiv \omega\pmod{\overline{\rho}}$, Proposition \ref{prop:cubicres}(b) and (d) give
	\[
	\left(\frac{\beta}{\overline{\rho}}\right)_3=\left(\frac{2}{\overline{\rho}}\right)_3=\omega^2.
	\]
	Taking conjugates in $\beta\equiv 2\pmod{\overline{\rho}}$ yields $\overline{\beta}\equiv 2\pmod \rho$, and therefore
	\[
	b=N(\beta)=\beta\overline{\beta}\equiv 2\beta\pmod \rho,
	\qquad\text{so}\qquad
	\beta\equiv 4b\pmod \rho.
	\]
	If $a\equiv 1\pmod 3$, Lemma \ref{lem:ab} gives $b\equiv \pm 1\pmod 7$, so $\beta\equiv \pm 4\pmod \rho$ and
	\[
	\left(\frac{\beta}{\rho}\right)_3=\left(\frac{\pm 4}{\rho}\right)_3=\left(\frac{2}{\rho}\right)_3^2=\omega^2.
	\]
	If $a\equiv 0\pmod 3$, then $b\equiv \pm 2\pmod 7$, so $\beta\equiv \pm 8\equiv \pm 1\pmod \rho$ and
	\[
	\left(\frac{\beta}{\rho}\right)_3=1.
	\]
	In either case,
	\[
	\left(\frac{28}{\beta_0}\right)_3=\omega^2.
	\]
\end{proof}

\begin{lemma}\label{lem:valuation}
Let $p$ be a primary prime of $D$ such that $p\nmid 14$ and $\left(\frac{28}{p}\right)_3\neq 1$. Then
$
3 \mid v_p(7x^3+2y^3)
$
for all integers $x,y$ that are not both $0$.
\end{lemma}
\begin{proof}
	If $p\nmid 7x^3+2y^3$, then $v_p(7x^3+2y^3)=0$. Otherwise let
	\[
	m:=\min(v_p(x),v_p(y)),\qquad x=p^m x_1,\qquad y=p^m y_1,
	\]
	where at least one of $x_1,y_1 \in D$ is not divisible by $p$. Then
	\[
	7x^3+2y^3=p^{3m}(7x_1^3+2y_1^3).
	\]
	We claim that $p\nmid 7x_1^3+2y_1^3$. Indeed, if $p\mid x_1$ and $p\mid 7x_1^3+2y_1^3$, then $p\mid 2y_1^3$, hence $p\mid y_1$ because $p\nmid 2$, contradicting the choice of $x_1$ and $y_1$. If $p\nmid x_1$ and $p\mid 7x_1^3+2y_1^3$, then
	\[
	28x_1^3\equiv (-2y_1)^3\pmod p,
	\]
	so $28$ is a cube modulo $p$, contrary to $\left(\frac{28}{p}\right)_3\neq 1$. Therefore $p\nmid 7x_1^3+2y_1^3$, and hence $v_p(7x^3+2y^3)=3m$.
\end{proof}

	Now we can prove Theorem \ref{th:main}. 
	
\begin{proof}[Proof of Theorem \ref{th:main}]
	Assume, for contradiction, that $(x,y,z)\in\mathbb Z^3$ satisfies \eqref{eq:7x3p2y3m3z2m1}. By Lemma \ref{lem:ab} we have that $z\equiv \pm 3\pmod 7$, so after replacing $z$ by $-z$ if necessary we may assume that $z\equiv 3\pmod 7$.
	Let $a,b,\beta$ and $\beta_0$ be as in Lemmas \ref{lem:ab}--\ref{lem:symbol2}. Since $\left(\frac{28}{\beta_0}\right)_3=\omega^2\neq 1$ by Lemma \ref{lem:symbol2}, definition \eqref{def:chi} shows that some primary prime $p$ appears in the factorization of $\beta_0$ with exponent $e\geq 1$ such that
	\[
	\left(\frac{28}{p}\right)_3^e\neq 1.
	\]
	Because $p\mid \beta_0$, Lemma \ref{lem:symbol} gives $p\nmid 14$. Moreover, $\beta$ and $\overline{\beta}$ are coprime and $N(\beta)=b$, hence $e=v_p(b)$.
	Now $p\mid b$ implies $p\mid 3z^2+1$, and therefore $p\mid 7x^3+2y^3$. Since $p\nmid 7$, we have
	\[
	e=v_p(b)=v_p(7x^3+2y^3).
	\]
	By Lemma \ref{lem:valuation}, $e$ is divisible by $3$. But then
	\[
	\left(\frac{28}{p}\right)_3^e=1,
	\]
	because $\left(\frac{28}{p}\right)_3$ is a cube root of unity. This contradicts the choice of $p$ and $e$, and completes the proof.
\end{proof}

\begin{remark}
	The contradiction obtained above may be interpreted as an integral Brauer--Manin obstruction. We nevertheless presented the proof in terms of cubic reciprocity, which keeps the prerequisites minimal.
\end{remark}

\section{The new shortest open equations}\label{sec:next}

Theorem \ref{th:main} completes the resolution of Problem \ref{prob:cub} for all cubic equations of length $l < 12+\log_2 3 \approx 13.6$. It is therefore natural to investigate the cubic equations of length $l = 12+\log_2 3$. We developed a computer program that generates all cubic equations of this length. We then eliminated (i) equations that have no integer solutions because of the ``elementary'' obstructions described in \cite{MR4823864}, and (ii) equations for which we were able to find an integer solution. This left six cubic equations of this length: the equation
\begin{equation}\label{eq:og}
	2+4x^3+3xy^2+z^3=0,
\end{equation}
and the five equations listed in Table \ref{tab:cubshortest}. We observed that the same cubic reciprocity method used for equation \eqref{eq:7x3p2y3m3z2m1} can also be used to resolve equation \eqref{eq:og}.

\begin{theorem}\label{th:og}
	Equation \eqref{eq:og} has no integer solutions.
\end{theorem}
\begin{proof}
	Assume, for contradiction, that there exist integers \(x,y,z\) satisfying \eqref{eq:og}. A reduction modulo $4$ shows that $y$ is odd, while a reduction modulo $9$ shows that $x\equiv 1\pmod 3$, $3\nmid y$, and $3\mid z$. By replacing $y$ with $-y$ if necessary, we may assume that $y\equiv 1\pmod 3$.
	
	Set
	\[
	\alpha:=2x+y+2y\omega \in D.
	\]
	Then
	\[
	N(\alpha)
	=(2x+y)^2-(2x+y)(2y)+(2y)^2
	=4x^2+3y^2,
	\]
	and from \eqref{eq:og} we obtain
	\[
	z^3+2=-x(4x^2+3y^2)=-x\,N(\alpha)=-x\,\alpha\overline{\alpha}.
	\]
	Hence,
	$
	\alpha\mid z^3+2.
	$
	
	Next, $x \equiv y \equiv 1 \pmod 3$ implies that $\alpha\equiv 2\omega \pmod 3$.
	Define
	\[
	\beta:=\omega^2 \alpha.
	\]
	Then \(\beta\equiv 2\pmod 3\), so \(\beta\) is primary. Also, because $y$ is odd,
	$\alpha \equiv 1\pmod 2$, and
	$\beta=\omega^2 \alpha\equiv \omega^2 \pmod 2$, so $\beta$ is coprime to $2$.
	
	Because $\alpha\mid z^3+2$, and $\omega^2$ is a unit, we also have
	$
	\beta\mid z^3+2,
	$
	hence
	$
	(-z)^3\equiv 2\pmod{\beta},
	$
	which implies that
	\begin{equation}\label{betasymbol}
		\left(\frac{2}{\beta}\right)_3 = 1.
	\end{equation}
	
	On the other hand,
	$$
	\left(\frac{\beta}{2}\right)_3\equiv \beta^{(N(2)-1)/3} = \beta^{(4-1)/3}=\beta \equiv \omega^2 \pmod 2,
	$$
	which implies that
	$
	\left(\frac{\beta}{2}\right)_3=\omega^2\neq 1.
	$
	This, together with \eqref{betasymbol}, contradicts Theorem \ref{th:gencubic}, and proves that equation \eqref{eq:og} has no integer solutions.
\end{proof}

\begin{table}[ht]
\caption{\label{tab:cubshortest} The current shortest cubic equations for which Problem \ref{prob:cub} remains open, all of length $l = 12+\log_2 3 \approx 13.6$.}
	\begin{center}
		\begin{tabular}{ |c|c| }
			\hline
			Equation & Equation \\
			\hline\hline
			$y^2z+yz^2=x^3+x^2+3x-1$ & $2x^3+3xy^2+z^3+z^2+1=0$ \\
			\hline
			$y^2z+yz^2=3x^3+x^2+x-1$ & $(3x-1)y^2 + x z^2 = x^3-2$ \\
			\hline
			$y^2z+yz^2=6x^3+x^2+1$ & \\
			\hline
		\end{tabular}
		
	\end{center}
\end{table}

After equation \eqref{eq:og} is resolved in Theorem \ref{th:og}, the five equations listed in Table \ref{tab:cubshortest} are the only cubic equations of length $l \leq 12+\log_2 3$ for which Problem \ref{prob:cub} remains open. The reader is invited either to find an integer solution to any of these equations or to prove that none exists. 

We remark that the length defined in \eqref{eq:lengthdef} is not the only way to order Diophantine equations. For example, one may also order equations by \emph{size}, $H(P)$, defined in \cite{MR4823864} as
$$
H(P) = \sum_{i=1}^k \bigl(|a_i|\cdot 2^{d_i}\bigr),
$$
where $a_i$ and $d_i$ are the same as in \eqref{eq:lengthdef}. Note that $H(P)$ corresponds to $L(P)$ in \eqref{eq:lengthdef}, but with the product replaced by a sum. If equations are ordered by size $H$, the current smallest cubic equation for which Problem~\ref{prob:cub} remains open is 
$$
y^2z+yz^2=x^3+x^2+3x-1
$$
of size $H=35$, which is the first entry in Table~\ref{tab:cubshortest}. See \cite{grechuk2024systematic} for a continually updated list of the smallest and shortest open equations in various categories. Despite the non-uniqueness of ordering methods, we consider the length \eqref{eq:lengthdef} to be the most natural choice; it approximates the number of symbols needed to write down an equation as a sum of monomials, assuming the power symbol is not used (e.g., writing $y^2z$ as $yyz$) and all coefficients are written in binary.

\section*{Acknowledgments}
We are grateful to the referee for their valuable comments and suggestions, which helped to improve the quality of the paper.

\bibliography{sref} 
\bibliographystyle{plain}
		
\end{document}